\documentclass[a4paper,11pt]{article}
\usepackage[english]{babel}
\usepackage[latin1]{inputenc}
\usepackage[T1]{fontenc}
\usepackage{amsthm}
\usepackage{amsmath}
\usepackage{amssymb}
\usepackage{amscd}
\usepackage{xy}
\usepackage{mathrsfs}
\usepackage{braket}
\usepackage{xcolor}
\usepackage{hyperref}
\usepackage{graphicx,subfig,floatrow}
\usepackage[nottoc]{tocbibind}
\usepackage{tikz}
\usepackage{enumitem}
\usetikzlibrary{arrows,shapes,positioning,mindmap,trees,automata,decorations.markings}

\hypersetup{
	colorlinks=false,
	allbordercolors=white,
}

\theoremstyle{definition}
\newtheorem{definition}{Definition}[section]

\theoremstyle{plain}
\newtheorem{teo}[definition]{Theorem}
\newtheorem{prop}[definition]{Proposition}
\newtheorem{cor}[definition]{Corollary}
\newtheorem{lem}[definition]{Lemma}
\newtheorem{conj}[definition]{Conjecture}

\newtheorem*{constr}{Construction}

\newtheorem*{claim}{Claim}

\newcommand{\numberset}{\mathbb}
\newcommand{\N}{\numberset{N}}
\newcommand{\R}{\numberset{R}}
\newcommand{\Z}{\numberset{Z}}
\newcommand{\Q}{\numberset{Q}}
\newcommand{\ca}{\mathcal}

\DeclareMathOperator{\degree}{d}

\title{Edge colorings and circular flows on regular graphs}

\author{Davide Mattiolo\thanks{Department of Physics, Informatics and Mathematics,
				University of Modena and Reggio Emilia, Via Campi 213/b, 41126 Modena, Italy. Email:~davide.mattiolo@unimore.it}, 	
Eckhard Steffen\thanks{Paderborn Center for Advanced Studies and Institute for Mathematics, Paderborn University, 
Warburger Str. 100, 33098 Paderborn, Germany. Email:~es@upb.de}}

\date{}

\begin{document}
	\maketitle

\begin{abstract}
Let $\phi_c(G)$ be the circular flow number of a bridgeless graph $G$.
In \cite{Stef} it was proved that, for every $t \geq 1$, $G$ is a bridgeless $(2t+1)$-regular 
graph with $\phi_c(G) \in \{2+\frac{1}{t}, 2 + \frac{2}{2t-1}\}$ if and only if $G$ has a perfect matching $M$ such that $G-M$ is bipartite. This implies that $G$ is a class 1 graph. For $t=1$, all graphs with circular flow number bigger than 4 are class 2 graphs. 
We show for all $t \geq 1$, that 
$2 + \frac{2}{2t-1} = \inf \{ \phi_c(G)\colon G \text{ is a } (2t+1) \text{-regular class } 2 \text{ graph}\}$.
This was conjectured to be true in \cite{Stef}. Moreover we prove that $\inf\{ \phi_c(G)\colon G $ is a $ (2t+1)$-regular class $1$ graph with no perfect matching whose removal leaves a bipartite graph$ \} = 2 + \frac{2}{2t-1}$. We further disprove the conjecture that
every $(2t+1)$-regular class $1$ graph has circular flow number at most $2+\frac{2}{t}$.
\end{abstract}

\section{Introduction}

	In this paper we consider finite graphs $G$ with vertex set $V(G)$ and edge set $E(G)$. 
	A graph may contain multiple edges but no loops. A $k$-edge-coloring of a graph $G$ is a function $c\colon E(G)\to \{1,\dots,k\}$. A coloring $c$ is proper if $c(e_1)\ne c(e_2)$ for any two adjacent edges $e_1,e_2\in E(G)$. The chromatic index $\chi'(G)$ of $G$ is the minimum $k$ such that $G$ admits a proper $k$-edge-coloring and it is known that $\chi'(G)\in \{\Delta(G),\dots,\Delta(G)+\mu(G)\}$, where $\mu(G)$ denotes the maximum multiplicity of an edge in $G$ \cite{Vizing}. A graph $G$ is class $1$ if $\chi'(G)=\Delta(G)$ and class $2$ otherwise.
	
Let $r\ge2$ be a real number, a nowhere-zero $r$-flow in a graph $G$ is a pair $(D,f)$ where $D$ is an orientation of $G$ and 
$f\colon E(G)\to [1,r-1]$ is a function such that $\sum_{e\in E^+(v)} f(e) = \sum_{e\in E^-(v)} f(e)$ for each vertex of $v\in G$, 
where $E^+(v)$ and $E^-(v)$ denote the set of all outgoing and incoming arcs at $v$ respectively. The circular flow number of $G$ is
$$\inf\{ r | G \mbox{ has a nowhere-zero $r$-flow} \},$$ and it is denoted by $\phi_c(G)$. 
It was proved in \cite{GTZ} that, for every bridgeless graph, $\phi_c(G)\in \Q $ and it is a minimum. 
If $G$ has a bridge, then it does not admit any nowhere-zero flow. Tutte's $5$-Flow Conjecture \cite{Tutte} states that every bridgeless graph admits a nowhere-zero $5$-flow, and it is one of the most important conjectures in this field. 
It is well known that it is equivalent to its restriction to cubic graphs. A snark is a cyclically 4-edge connected cubic graph
with girth at least 5 and which does not admit a nowhere-zero 4-flow. 

Tutte (see \cite{Tutte2} and \cite{Tutte}) proved that a cubic graph $G$ has a nowhere-zero 3-flow if and only if  
$G$ is bipartite and that $G$ has a nowhere-zero 4-flow if and only if $G$ is a class 1 graph. In \cite{Stef_circ_flow}
it is shown that there is no cubic graph $H$ with $3 < \phi_c(H) < 4$. 
In \cite{Stef} these results are generalized to $(2t+1)$-regular graphs. 
	
	\begin{teo}[\cite{Stef}]\label{teo:bipartite}
		A $(2t+1)$-regular graph $G$ is bipartite if and only if $\phi_c(G)= 2+ \frac{1}{t}$. Furthermore, if $G$ is not bipartite, then $\phi_c(G)\ge 2+ \frac{2}{2t-1}$.
	\end{teo} 

	\begin{teo}[\cite{Stef}]\label{teo:1-factor_bipartite}
		A $(2t+1)$-regular graph $G$ has a $1$-factor $M$ such that $G-M$ is bipartite if and only if $\phi_c(G)\le 2+\frac{2}{2t-1}$.
	\end{teo}


In \cite{Stef} it is further shown that, different from the cubic case, for every $t\ge 2$, there is no flow number that separates 
$(2t+1)$-regular class $1$ graphs from class $2$ ones. In particular Theorem \ref{teo:1-factor_bipartite} implies that a $(2t+1)$-regular graph $G$ having $\phi_c(G)\le2+ \frac{2}{2t-1}$ is class $1$ and it was conjectured that this is the biggest flow number $r$ such that every $(2t+1)$-regular graph 
$H$ with $\phi_c(H)\le r$ is class $1$. 
Let $ \Phi^{(2)}(2t+1) :=\inf \{ \phi_c(G)\colon G \text{ is a } (2t+1) \text{-regular class } 2 \text{ graph}\}$.
	
	\begin{conj}[\cite{Stef}]\label{conj:Stef_inf}
		For every integer $t\ge1$ $$ \Phi^{(2)}(2t+1) = 2+\frac{2}{2t-1}. $$
	\end{conj}
	

We prove Conjecture \ref{conj:Stef_inf} in Section \ref{PROOF}. Moreover, let us define $\ca G_{2t+1}:=\{G \colon G$ is a $(2t+1)$-regular class $1$ graph such that there is no perfect matching $M$ of $G$ such that $G-M$ is bipartite$ \}$. Consider the following parameter: $$ \Phi^{(1)}(2t+1):=\inf\{\phi_c(G)\colon G \in \ca G_{2t+1} \}. $$ In Section \ref{PROOF} we further prove that $\Phi^{(1)}(2t+1)=2+\frac{2}{2t-1}$, for all positive integers $t$.

If a graph $G$ has a small odd edge cut, say of cardinality $2k+1$, then $\phi_c(G) \geq 2 + \frac{1}{k}$.
Recall that an $r$-graph is an $r$-regular graph $G$ such that $|\partial_G(X)|\ge r $, for every $X\subseteq V(G)$ with $|X|$ odd, where $|\partial_G(X)|$ denotes the set of edges of $G$ with exactly one end in $X$. The circular flow number
of the complete graph $K_{2t+2}$ on $2t+2$ vertices is $2 + \frac{2}{t}$ \cite{Stef_circ_flow} and $K_{2t+2}$
is a class 1 graph.  
	
	\begin{conj}[\cite{Stef}]\label{conj:Stef_class1}
		Let $G$ be a $(2t+1)$-regular class $1$ graph. Then $\phi_c(G)\le 2 + \frac{2}{t}$.
	\end{conj}

	\begin{conj}[\cite{Stef}]\label{conj:Stef_r-graph}
		Let $G$ be a $(2t+1)$-graph. Then $\phi_c(G)\le 2 + \frac{2}{t}$.
	\end{conj}

	Since $(2t+1)$-regular class $1$ graphs are $(2t+1)$-graphs Conjecture \ref{conj:Stef_r-graph} implies Conjecture \ref{conj:Stef_class1}. We show that both these conjectures are false by constructing $(2t+1)$-regular class $1$ graphs with circular flow number greater than $2+\frac{2}{t}$. The construction of the counterexamples relies
on a family of counterexamples to Jaeger's Circular Flow Conjecture \cite{Jaeger} which was 
given by Han, Li, Wu, and Zhang in \cite{Zhang1}. 
	
\section{Circular flow number of $(2t+1)$-regular graphs.} \label{PROOF}
	
Let $G$ and $H$ be two connected cubic graphs with at least $6$ vertices. Let $G'= G - \{v_1v_2, v_3v_4\}$,
where $\{v_1v_2, v_3v_4\}$ is a matching of $G$. Let $u,w$ be two adjacent vertices of $H$ and let $H'=H-\{u,w\}$. 
Furthermore, let $u_1,u_2$ and $w_1,w_2 \in V(H)$ be the neighbors $u$ and $w$ in $G$, respectively,
which are elements of $V(H')$. The dot product $G \cdot H$ is defined to be the graph with 
	$V(G\cdot H) = V (G) \cup V (H')$, and $E(G \cdot H) = E(G') \cup E(H') \cup \{v_1u_1 , v_2u_2 , v_3w_1 , v_4w_2 \}.$
	
	The following is a well known result that goes back to Izbicki \cite{Izbicki_66}.
	
	\begin{lem}[Parity Lemma]
		Let $G$ be a $(2t+1)$-regular graph of class $1$ and $c\colon E(G)\to \{1,2,\dots,2t+1\}$ a proper edge-coloring of $G$. Then, for every edge-cut $C\subseteq E(G)$ and color $i$, the following relation holds $$ |C\cap c^{-1}(i)| \equiv |C| \mod 2. $$
	\end{lem}

	Let $G$ be a graph and let $M\subseteq E(G)$. We denote by $G+M$ the graph obtained by adding a copy of $M$ to $G$. Such a graph has vertex set $V(G+M)=V(G)$ and edge set $E(G+M)=E(G)\cup M'$, where $M'$ is a copy of $M$.
	Let $G_1$ and $G_2$ be cubic graphs having perfect matching $M_1$ and $M_2$ respectively. Let $G$ be a dot product $G_1\cdot G_2$ where we remove from $G_1$ two non-adjacent edges $e_1,e_2 \in E(G_1- M_1)$ and from $G_2$ two adjacent vertices $x,y$ such that $xy\in M_2$. Then we say that $G$ is an $(M_1,M_2)$-dot-product of $G_1$ and $G_2$.
	
	Moreover, let $H$ be a cubic graph with a perfect matching $M_3$ such that, for all positive integers $t$, $H+(2t-2)M_3$ is a 
	$(2t+1)$-regular class $1$ (resp. class $2$) graph. Then we say that $H$ has the $M_3$-class-$1$ (resp. $M_3$-class-$2$) property.
		
\begin{lem}\label{lem:class 2}
	For $i=1,2$, let $G_i$ be a cubic graph having the $M_i$-class-$2$ property, where $M_i$ is  a perfect matching of $G_i$. Moreover let $G$ be an $(M_1,M_2)$-dot-product of $G_1$ and $G_2$ and $x,y\in V(G_2)$ the two adjacent vertices that have been removed from $G_2$ when constructing $G$. Then $M=M_1\cup M_2\setminus \{xy\}$ is a perfect matching of $G$ and $G$ has the $M$-class-$2$ property.
\end{lem}

\begin{proof} 
Let $e_1 = v_1v_2$, $e_2=v_3v_4$ be the edges that have been removed from 
$G_1$ in order to obtain $G$.
Define $H=G+(2t-2)M$ and $H_i=G_i+(2t-2)M_i$, $i \in \{1,2\}$, and let $a_1,a_2,a_3,a_4$ be the added edges incident 
to $v_1,v_2,v_3$ and $v_4$ respectively. Then $C = \{a_1,a_2,a_3,a_4\}$ is a 4-edge-cut in $H$
which separates $H[V(G_2)-\{x,y\}]$ and $H[V(G_1)]$.

Suppose to the contrary that $H$ is a class $1$ graph. By the Parity Lemma, either $C$ intersects only one color class, or it intersects two color classes in exactly two edges each. Moreover, if $c(a_1)=c(a_2)$, then $c(a_3)=c(a_4)$ and a $(2t+1)$-edge-coloring is defined naturally on $H_1$ by the coloring of $H$ in contradiction to the fact that $H_1$ is a class $2$ graph. Therefore, $c(a_1)\ne c(a_2)$, and so 
$\{c(a_3), c(a_4)\} = \{c(a_1),c(a_2)\}$. In this case a $(2t+1)$-edge-coloring is naturally defined on $H_2$ by the coloring of $H$ leading to a contradiction again.
\end{proof}

	For the following result we will need the concept of balanced valuations, see \cite{Bondy} and \cite{Jaeger_bal_val}. Let $G$ be a graph, a balanced valuation of $G$ is a function $\omega \colon V(G)\to \R$ such that 
	$|\sum_{v\in X} \omega(v)| \le |\partial_G(X)|$, for every $X\subseteq V(G)$. Balanced valuations and nowhere-zero flows are
	equivalent concepts as the following theorem shows (this formulation is given in \cite{Stef_circ_flow}).

	\begin{teo}[\cite{Jaeger_bal_val}] \label{bv}
		Let $G$ be a graph. Then $G$ has a nowhere-zero $r$-flow if and only if there is a balanced valuation $\omega\colon V(G)\to \R$ of $G$ such that, for every $v\in X$ there is an integer $k_v$ such that $k_v \equiv \degree_G(v) \mod{2} $ and $\omega(v)=k_v\frac{r}{r-2}$.
	\end{teo}

If $G$ has a nowhere-zero $r$-flow, then $G$ has always an orientation such that all flow values are positive. 
Thus, if $G$ is cubic, then each vertex has either one or two incoming edges. Hence, $V(G)$ is naturally partitioned 
into two subsets of equal cardinality $V(G) = \ca B \cup \ca W$.  Moreover, we say that $v$ is black (resp.~white) if $v\in\ca B$ (resp. $v\in\ca W$). The balanced valuation $\omega$ of $G$ corresponding to the all-positive 
nowhere-zero $r$-flow $(D,f)$ is defined as follows: $\omega(v)=\frac{r}{r-2}$ if $v$ is black and 
$\omega(v)=-\frac{r}{r-2}$ if $v$ is white. Finally, for $X\subseteq V(G)$ we define $b_X = |X\cap \ca B|$ and $w_X=|X\cap \ca W|$. We call the partition $(\ca B, \ca W)$ of $V(G)$ an $r$-bipartition of $G$, see for example \cite{Flows_bisections} for the study of such partitions in cubic graphs. 

	\begin{lem}\label{lem:asymptotic} \label{lemma}
		Let $i\in\{1,2\}$, and $\{G_n : n\in \N \}$ be a family of cubic class $2$ graphs such that for each $n \ge 1$:
		\begin{itemize}
		\item $G_n$ has a $r_n$-bipartition $(\ca B_n, \ca W_n)$ with $r_n\in(4,5)$;
		\item $G_n$ has a perfect matching $M_n$ with the following properties:
		\begin{itemize}
			\item $G_n$ has the $M_n$-class-$i$-property;
			\item if $ab\in M_n$, then $a\in\ca B_n$ if and only if $b\in \ca W_n$.
		\end{itemize}
		\end{itemize}
	
		If $\lim_{n \to \infty} r_n = 4$, then $\Phi^{(i)}(2t+1)=2+\frac{2}{2t-1}\text{, for every integer }t\ge1$.
		
	\end{lem}

	\begin{proof}
		Fix an integer $t\ge1$ and let $H_n := G_n + (2t-2)M_n$. Since $G_n$ has the $M_n$-class-$i$ property $\{H_n\colon n \in \N\}$ is an infinite family of $(2t+1)$-regular class $i$ graphs.
		
		 By Theorem \ref{bv}, $|\partial_{G_n}(X)|\ge \frac{r_n}{r_n - 2}|b_X-w_X|$, for every $X\subseteq V(G_n)$. Let $Y\subseteq V(H_n)$. Since $M_n$ pairs black and white vertices of $H_n$ we have that $d = |M_n \cap \partial_{H_n}(Y)| \ge |b_Y-w_Y|. $ Therefore, for every $Y\subseteq V(H_n)$, we get the following inequalities:
		 
		 $$ |\partial_{H_n} (Y)| \ge \frac{r_n}{r_n - 2}|b_Y-w_Y| + (2t-2)d \ge (\frac{r_n}{r_n - 2}+2t-2)|b_Y-w_Y|. $$
		
		Hence, $H_n$ has a nowhere-zero $(2+ \frac{2(r_n-2)}{r_n+(2t-3)(r_n-2)})$-flow. Notice that, if $i=1$, then $H_n\in \ca G_{2t+1}$, for every $n$, because $G_n$ is a class $2$ cubic graph, and so it cannot have a $1$-factor whose removal gives rise to a bipartite graph. On the other hand if $i=2$, then $H_n$ is class $2$. Therefore, since the sequence $\{r_n\}_{n\in\N}$ tends to $4$ from above, we have that $$ \Phi^{(i)}(2t+1) \le \lim_{n\to \infty} \Bigl(2+ \frac{2(r_n-2)}{r_n+(2t-3)(r_n-2)}\Bigr) = 2+\frac{2}{2t-1}, $$
		and thus, equality holds from Theorem \ref{teo:1-factor_bipartite}.
	\end{proof}

\subsection{A family of snarks fulfilling the hypothesis of Lemma \ref{lem:asymptotic}}
\subsubsection*{Class $1$ regular graphs}

Consider the family of Flower snarks $\{J_{2n+1}\}_{n\in \N}$, introduced in \cite{Isaacs}.
The Flower snark $J_{2n+1}$ is the non-$3$-edge-colorable cubic graph having: 
\begin{itemize}
	\item vertex set $V(J_{2n+1})= \{a_i,b_i,c_i,d_i\colon i \in \Z_{2n+1}\}$
	\item edge set $E(J_{2n+1})= \{ b_ia_i,b_ic_i,b_id_i,a_ia_{i+1},c_id_{i+1},c_{i+1}d_i\colon i \in \Z_{2n+1}\}$
\end{itemize}

The following Lemma holds true. Since its proof requires some case analysis and lies outside the intent of this section we omit it here and add it in the appendix.

\begin{lem}\label{lem:Flower}
	Let $M$ be a $1$-factor of $J_{2n+1}$. Then $J_{2n+1}+M$ is a class $1$ $4$-regular graph.
\end{lem}

\begin{teo}\label{teo:prop_Flower}
	The graph $J_{2n+1}$ has a $(4 + \frac{1}{n})$-bipartition $(\ca B_n, \ca W_n)$ and a perfect 
	matching $M_n$ such that:
	\begin{itemize}
		\item $J_{2n+1}$ has the $M_n$-class-$1$-property;
		\item for all $xy\in M_n$, $x\in B_ n$ if and only if $y\in W_n$.
	\end{itemize}
\end{teo}

\begin{proof}
	
	We construct explicitly a nowhere-zero $(4+\frac{1}{n})$-flow $(D_n,f_n)$ in $J_{2n+1}$ as sum of an integer $4$-flow $(D,f)$ with exactly one edge having flow value $0$ and $n$ flows $(D'_1,f_1'),\dots,(D_n',f_n')$ having value $\frac{1}{n}$ each on a different circuit.
	
	Define $(D,f)$ on the directed edges of $J_{2n+1}$ as follows, when we write an edge connecting two vertices $u,v$ in the form $uv$ we assume it to be oriented from $u$ to $v$ in the orientation $D$:
	
	\begin{itemize}
		\item $f(a_0b_0)=0$ and $f(b_0c_0)=f(d_0b_0)=2$;
		\item $f(b_ia_i)=f(a_{i+1}b_{i+1})=1$, for all $i\in\{1,3,\dots,2n-1\}$;
		\item $f(b_ic_i)=2$ and $f(b_{i+1}c_{i+1})=3$, for all $i\in\{1,3,\dots,2n-1\}$;
		\item $f(d_ib_i)=3$ and $f(d_{i+1}b_{i+1})=2$, for all $i\in\{1,3,\dots,2n-1\}$;
		\item $f(a_ia_{i+1})=f(c_{i+1}d_i)=1$, for all $i\in \{0,2,\dots,2n\}$;
		\item $f(a_ia_{i+1})=f(c_{i+1}d_i)=2$, for all $i\in\{1,3,\dots,2n-1\}$;
		\item $f(c_id_{i+1})= 1$, for all $i\in \Z_{2n+1}$;
	\end{itemize}
	
	For $j \in \{1,\dots,n \}$ let $(D'_j,f'_j)$ be the flow on the directed circuit 
	$C_j= a_0b_0c_0 d_1\dots d_l c_{l+1} d_{l+2}\dots d_{2j-1} b_{2j-1} a_{2j-1} a_{2j}a_{2j+1}\dots a_0$ (where $l<j$ and $l$ odd), 
	with $f_j'(e)=\frac{1}{n}$ if $e\in C_j$ and $f_j'(e)=0$ otherwise.
	
	The sum $(D,f)+\sum_{i = 1}^{n}(D_i',f_i')$ gives a nowhere-zero $(4+\frac{1}{n})$-flow $(D_n,f_n)$ in $J_{2n+1}$. Let $(\ca B_n,\ca W_n)$ be the bipartition induced by such a flow and consider the $1$-factor $M_n=\{ a_ib_i, c_{i+1}d_i\colon i\in \Z_{2n+1} \}$. Notice that $D_n=D$, and for all $xy\in M_n$: $x\in\ca B_ n$ if and only if $y\in\ca W_n$. By Lemma \ref{lem:Flower}, $J_{2n+1}$ has the $M_n$-class-$1$ property.
\end{proof}

\subsubsection*{Class $2$ regular graphs}
Let $P$ denote the Petersen graph. We recall now the following result, which follows from Theorem 3.1 of \cite{GrundSteff}.
	
\begin{lem}\label{lem:P_class2}
		Let $M_1,\dots,M_k$ be perfect matchings of $P$. Then $P+\sum_{i=1}^{k}M_i$ is a $(k+3)$-regular class 2 graph.
\end{lem}

\subsubsection*{Construction of $\ca G =\{ G_n \colon n\in\N\}$} \label{construction}


$G_1$: The graph $G_1$ is the Blanu\v sa snark, Figure \ref{Fig:Blanusa}.

$G_{n+1}= G_n \cdot G_1$: The dot product of these two graphs will be carried out as follows. 
If $v\in V(G_1)$, then the vertex $v^{i}\in V(G_n)$ corresponds to the vertex $v$ of the $i$-th copy of $G_1$ that has been added in order to construct $G_{n}$. Consider the bold circuit $C=x_0x_1\dots x_8\subseteq G_1$ as depicted in Figure \ref{Fig:Blanusa_cycles}. Delete the vertices $x_0,x_1$ of $G_1$ and edges $x_4^{n}x_5^{n},x_7^{n}x_8^{n}$ from $G_n$. Perform the dot product $G_n\cdot G_1$ by adding the edges $x_4^nx_8,x_5^ny_0,x_7^ny_1$ and $x_8^nx_2$, where $y_0,y_1$ are vertices of $G_1$ which are not in $C$ and adjacent to $x_0,x_1$ respectively, see Figure \ref{Fig:Blanusa_cycles}. The snark $G_2$ is depicted in Figure \ref{Fig:ind_step}.

\begin{teo}\label{teo:Properties_Gn} Let $n\in\N$ and $G_n\in \ca G$.
	The graph $G_n$ has a $(4 + \frac{1}{n+1})$-bipartition $(\ca B_n, \ca W_n)$ and a perfect 
	matching $M_n$ such that:
	\begin{itemize}
		\item $G_n$ has the $M_n$-class-$2$-property;
		\item for all $xy\in M_n$, $x\in B_ n$ if and only if $y\in W_n$.
	\end{itemize}
\end{teo}

\begin{proof}
	First we show that for every $n$ there is a nowhere-zero $(4 + \frac{1}{n+1})$-flow in $G_n$. We argue by induction over $n\in \N$.
	Fix on $G_1$ the $4$-flow $(D_1,f_1)$ as depicted in Figure \ref{Fig:Blanusa}.
	When we write $D_1^{-1}$ we will refer to the orientation constructed by reversing each edge in $D_1$, similarly $D_1^1$ will be the orientation $D_1$. A nowhere-zero $(4 + \frac{1}{2})$-flow in $G_1$ can be constructed by adding $\frac{1}{2}$ along the two directed circuits $C_1,C_2$ depicted in Figure \ref{Fig:Blanusa_cycles}. Indeed they have the following two properties: \begin{enumerate} [label=\textbf{P.\arabic*}]
		\item \label{P1} the unique edge having flow value $0$ belongs to all circuits;
		\item \label{P2} every edge with flow value $3$ belongs to at most one of the circuits.
	\end{enumerate} 
	
	Notice also that $f_1(x_7x_8)=1$ and $f_1(x_4x_5)=2$. Moreover there is a unique circuit in $\{C_1,C_2\}$ containing the path $x_4\dots x_8$ and the other one does not intersect it.
	
	Now we proceed with the inductive step. By the inductive hypothesis there is a $4$-flow $(D_n,f_n)$ in $G_n$ having a unique edge with flow value $0$ and $n+1$ directed circuits $\{C_1,\dots,C_{n+1}\}$ in $D_n$ satisfying properties \ref{P1} and \ref{P2}. 
	It holds $f_n(x_7^nx_8^n)=1$, $f_n(x_4^nx_5^n)=2$. Furthermore, there is a unique circuit $C\in\{C_1,\dots,C_{n+1}\}$ containing the path $\tilde{P}=x_4^n\dots x_8^n$ and such that no other circuit intersects $\tilde{P}$. If $n$ is odd, then $\tilde{P}$ is a directed path in $D_n$, if $n$ is even, then $x_8^nx_7^n\dots x_4^n$ is a directed path in $D_n$.
	
	Let $H_n= G_n-\{x_4^nx_5^n,x_7^{n}x_8^{n}\}$ and $H'=G_1-\{x_0,x_1\}$. Then $G_{n+1}$ is constructed by adding edges $x_4^nx_8,x_5^ny_0,x_7^ny_1$ and $x_8^nx_2$.
	Let $(D_{n+1},f_{n+1})$ be the unique $4$-flow in $G_{n+1}$ such that
	\begin{itemize}
		\item $D_{n+1}|_{H_n}=D_n|_{H_n}$ and $D_{n+1}|_{H'}=D_1^{(-1)^{n}}|_{H'}$;
		\item $f_{n+1}|_{H_n}=f_n|_{H_n}$ and $f_{n+1}|_{H'}=f_1|_{H'}$.
	\end{itemize}

We show that there exists a set of $\tilde{\ca C}$ of $n+2$ circuits satisfying properties \ref{P1} and \ref{P2}. In particular, we are going to construct two circuits out of $C$. First notice that there are exactly two paths $\tilde{P}_1=x_8^{n+1}w_1\dots w_tx_2^{n+1}$ and $\tilde{P}_2=x_8^{n+1}x_7^{n+1}x_6^{n+1}x_5^{n+1}x_4^{n+1}x_3^{n+1}x_2^{n+1}$ in $H'\subseteq G_{n+1}$, which are directed in $D_{n+1}|_{H'}$ and such that $\tilde{P}_1\cap \tilde{P}_2 = x_2^{n+1}x_3^{n+1}$, for some vertices $w_1,\dots,w_t \in V(G_{n+1})$ (see Figure \ref{Fig:ind_step} for an example in the case of $n=1$). In particular, if $n$ is odd, then $P_1$ and $P_2$ are 
both directed from $x_8^{n+1}$ to $x_2^{n+1}$ and vice versa if $n$ is even. We can suppose without loss of generality that $C=v_0v_1\dots v_k \tilde{P}=v_0\dots v_k x_4^n\dots x_8^n$ in $G_n$. Define $\tilde{C}_i$ to be the circuit $v_0\dots v_kx_4^n\tilde{P}_ix_8^n$, $i \in \{1,2\}$. It follows by construction that $C_1$ and $C_2$ are both directed circuits in $D_{n+1}$. Therefore, the family $\tilde{\ca C} = (\ca C \setminus \{C\}) \cup \{\tilde{C}_1,\tilde{C}_2\}$ consists of $n+2$ circuits satisfying properties \ref{P1} and \ref{P2} and so a nowhere-zero $(4+\frac{1}{n+2})$-flow can be constructed in $G_{n+1}$.
	
	Now we show that for every $n$ there is a perfect matching $M_n$ of $G_n$ satisfying the statement. We argue again by induction. Choose as a perfect matching $M_1$ of $G_1$, which is indicated by bold edges in Figure \ref{Fig:Blanusa}. Consider two copies of the Petersen graph $P_1,P_2$ together with a perfect matching $N_i$ of $P_i$, $i \in \{1,2\}$. Recall that $G_1$ is constructed by 
	performing an $(N_1,N_2)$-dot-product $P_1\cdot P_2$. In particular we can choose $N_1$, $N_2$ and perform the dot product in such a way that $M_1 = N_1\cup N_2 \setminus \{x'y'\}$, where $x',y'$ are the vertices we removed from $P_2$ in order to perform the dot product. Therefore, by Lemmas \ref{lem:class 2} and \ref{lem:P_class2}, it follows that $G_1$ has the $M_1$-class-$2$-property. Figure \ref{Fig:Blanusa} shows that the chosen $(4 + \frac{1}{2})$-flow in $G_1$ and the perfect matching $M_1$ are related in the following way: let $\ca B_1 \cup \ca W_1$ be the $(4 + \frac{1}{2})$-bipartition of $V(G_1)$, then for all $xy\in M_1$, $x\in \ca B_1$ if and only if $y\in \ca W_1$. Therefore, the statement is true for $n=1$.	
Notice that $x_0x_1\in M_1$ and that $x_4x_5,x_7x_8\notin M_1$.
	
	For the inductive step we assume that $G_n$ has a perfect matching $M_n$ fulfilling the induction hypothesis and $x_4^nx_5^n,x_7^nx_8^n\notin M_n$. There is a unique perfect matching $M_{n+1}$ of $G_{n+1}$ such that $M_{n+1} \cap E(H_n) = M_n$ and $M_{n+1}\cap E(H')=M_1\setminus \{ x_0x_1 \}$. Thus, by Lemma \ref{lem:class 2}, we get that $G_{n+1}$ has the $M_{n+1}$-class-$2$-property. 
	
Define $\ca B_{n+1} = \ca B_n \cup \ca B$ and $\ca W_{n+1} = \ca W_n \cup \ca W$, where $(\ca B, \ca W)$ is the bipartition induced by $(D_1^{(-1)^n},f_1)$ in $G_1$. The bipartition $(\ca B_{n+1}, \ca W_{n+1})$ of $V(G_{n+1})$ is a 
$(4 + \frac{1}{n+1})$-bipartition. Since both $M_n = M_{n+1} \cap E(H_n)$ and $M_1 \setminus \{ x_0x_1 \} = M_{n+1}\cap E(H')$ pair black and white vertices, it follows that $M_{n+1}$ pairs black and white vertices too. Notice that $x_4^{n+1}x_5^{n+1},x_7^{n+1}x_8^{n+1}\notin M_{n+1}$, this concludes the inductive step.
\end{proof}

From Lemma \ref{lem:asymptotic} and Theorems \ref{teo:prop_Flower} and \ref{teo:Properties_Gn} we deduce the following corollary.

\begin{cor}\label{Cor:inf_class2}
	For every $t\ge 2$ and $i\in \{1,2\}\colon  \Phi^{(i)}(2t+1) = 2 + \frac{2}{2t-1}. $
\end{cor}

	\begin{figure}
		\centering
		\includegraphics[scale=0.7]{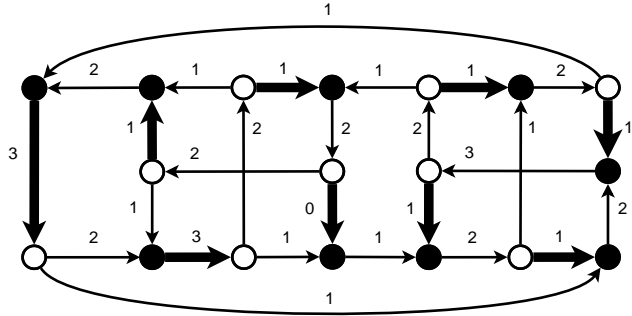}
		\caption{A $4$-flow in the Blanu\v sa snark $G_1$ having just one edge with flow value $0$. The perfect matching consisting of all bold edges pairs black vertices with white vertices.}\label{Fig:Blanusa}
	\end{figure}

	\begin{figure}
		\centering
		\includegraphics[scale=0.7]{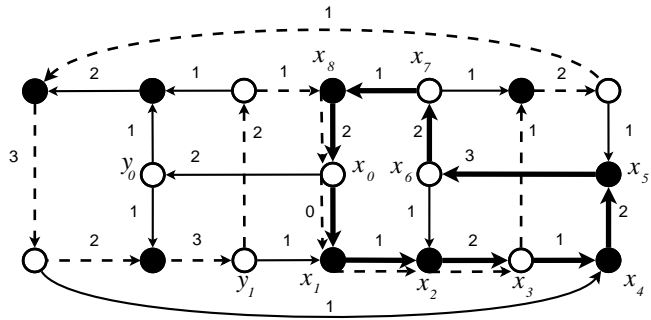}
		\caption{A nowhere-zero $(4+\frac{1}{2})$-flow in the Blanu\v sa snark $G_1$ can be constructed by adding $\frac{1}{2}$ along the bold and dotted circuits.}\label{Fig:Blanusa_cycles}
	\end{figure}

	\begin{figure}
		\centering
		\includegraphics[scale=0.7]{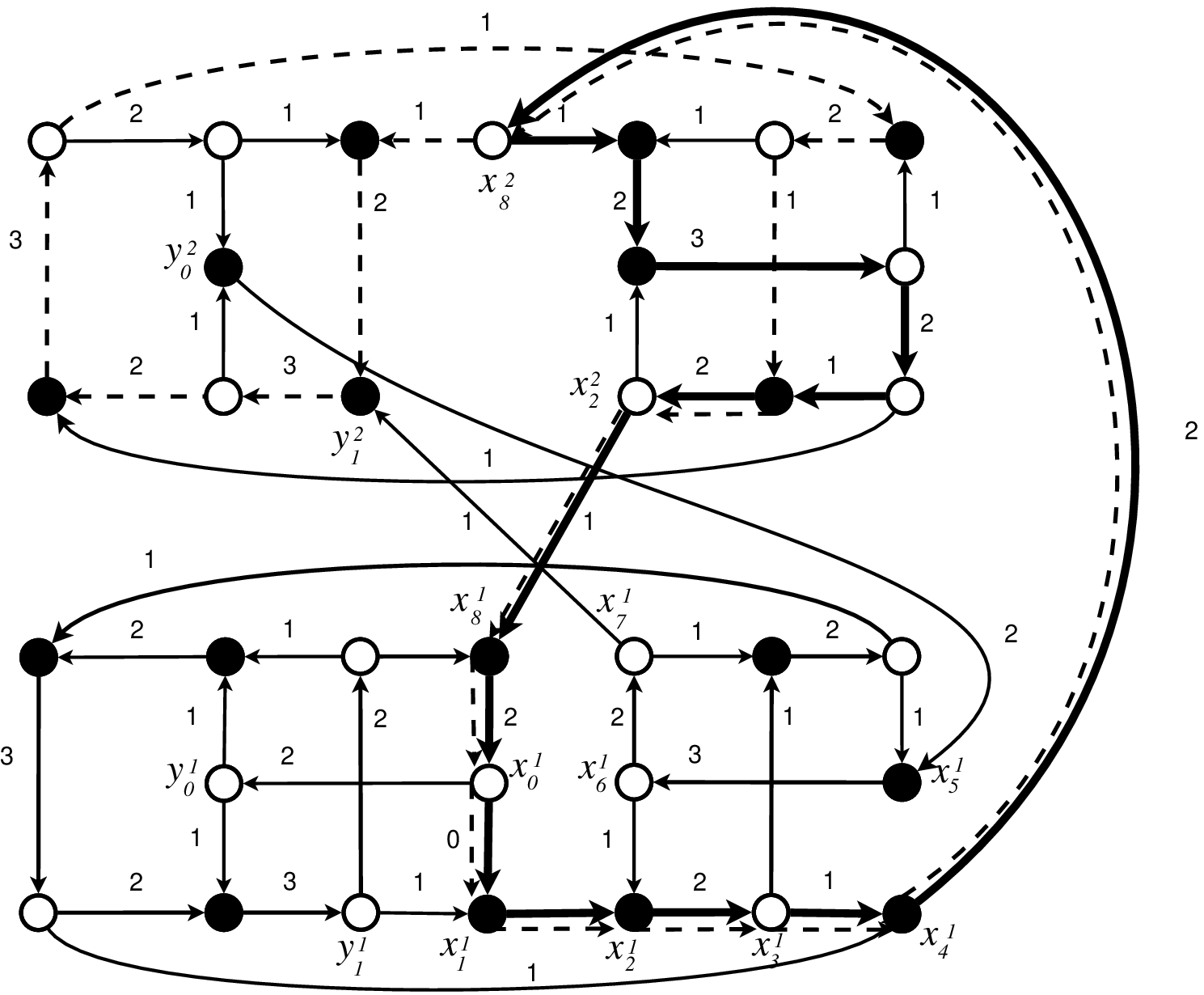}
		\caption{Construction of two more circuits (dotted and bold ones) in $G_2$.}\label{Fig:ind_step}
	\end{figure}

\section{Counterexamples to Conjecture \ref{conj:Stef_class1}}

In \cite{Zhang1} Jaeger's Circular Flow Conjecture \cite{Jaeger} is disproved. For each $p \geq 3$
they constructed a $4p$-edge-connected graph which does not admit a nowhere-zero $(2 + \frac{1}{p})$-flow. 
We will use this family of graphs to construct $(2k+1)$-regular class 1 graphs which do not admit a $(2 + \frac{2}{k})$-flow, for suitable values of $k$. 
We start with summarizing the construction of \cite{Zhang1}. 

\begin{constr}
	Let $p\ge 3$ be an integer and let $\{v_1,v_2,\dots,v_{4p}\}$ be the vertex set of the complete graph $K_{4p}$.
	
	\begin{itemize}
		\item[i.] Construct the graph $G_1$ by adding an additional set of edges $T$ such that $V(T)=\{v_1,v_2,\dots,v_{3(p-1)}\}$ and each component of the edge-induced subgraph $G_1[T]$ is a triangle.
		
		\item[ii.] Construct the graph $G_2$ from $G_1$ by adding two new vertices $z_1$ and $z_2$, adding one edge $z_1z_2$, adding $p-2$ parallel edges connecting $v_{4p}$ and $z_i$ for both $i \in \{1,2\}$, and adding one edge $v_iz_j$ for each $3p-2\le i \le 4p-1$ and $j \in \{1,2\}$.
		
		\item[iii.] Consider $4p+1$ copies $G_2^1,\dots, G_2^{4p+1}$ of $G_2$. If $v\in V(G_2)$, then we write $v^i$ to refer to the vertex $v$ of the $i$-th copy of $G_2$. Construct the graph $M_p$ the following way. For every $i\in \{1, \dots, 4p+1\}$ 
identify $z_2^i$ with $z_1^{i+1}$ and call this new vertex $c_{i+1}$, where we take sums modulo $4p+1$. Finally add a new vertex $w$ and all edges of the form $wc_i$ for all $i\in \{1, \dots,  4p+1\}$.
	\end{itemize}	
\end{constr}

For every $i \in \{1, \dots,  4p+1\}$ we have 
$d_{M_p}(v_{4p}^i) = 4p-1 + 2(p-2) = 6p - 5$, $d_{M_p}(c_i) = 2(p-2+ 4p-1 - (3p-2) + 1) + 3 = 4p+3$, and
all other vertices have degree $4p + 1$.

\begin{teo}[\cite{Zhang1}]
	For all $p \geq 3\colon \phi_c(M_p)>2+\frac{1}{p}.$
\end{teo}

Let $k = 2p$. The graph $M_p$ defined in previous section does not have a nowhere-zero $(2+\frac{2}{k})$-flow. We will use 
$M_p$ in order to construct a $(2k+1)$-regular graph of class $1$ which does not admit a nowhere-zero $(2+\frac{2}{k})$-flow, thus disproving Conjecture \ref{conj:Stef_class1}.

We consider odd integers $p\ge3$, say $p = 2t+1$.

We say that an expansion of a vertex $v$ of a graph $G$ to a graph $K$ is obtained by the replacement of $v$ by $K$ and by 
adding $d_G(v)$ edges with one end in $V(G-v)$ and the other end in $V(K)$.

\begin{prop}
	Let $H$ be a graph obtained by expanding a vertex of $G$. Then $\phi_{c}(H)\ge \phi_{c}(G).$
\end{prop}

\subsection*{Construction of $M_p'$}

The copy of $K_{4p}$ which is used in the construction of $G_2^i$ is denoted by $K_{4p}^i$.
Construct the graph $M_p'$ by expanding each vertex $v_{4p}^i $ of $ K_{4p}^i$ in $M_p$ to a vertex $x^{i}$ of degree $4p+1$
and $p-3$ divalent vertices, where $x^{i}$ is adjacent to every vertex of $V(K_{4p}^i) \setminus \{v_{4p}^i\}$ and to $c_i$ and $c_{i+1}$ and each divalent vertex is adjacent to both $c_i$ and $c_{i+1}$. After that suppress the divalent vertices. Note that the construction can also be seen as a edge splitting at $v_{4p}^i$. We have $d_{M_p'}(c_i) = 4p+3$ for all $i \in \{1, \dots, 4p+1\}$
and all other vertices of $M_p'$ have degree $4p+1$. Notice that $M'_p$ remains a bridgless graph.

\begin{lem}\label{lem:M_p'}
The graph $M_p'$ admits a $(4p+1)$-edge-coloring $c$ such that for all $i \in \{0, \dots, 4p\}$ and all $v \in V(M_p'):$
$|c^{-1}(i)\cap \partial (v)|$ is odd. Furthermore, $\phi_c(M_p') > 2+\frac{1}{p}$.
\end{lem}

\begin{proof}
	All operations performed in order to construct $M_p'$ do not decrease the circular flow number of graphs. Thus $\phi_c(M_p')\ge \phi_c(M_p) >2+\frac{1}{p}$.
	
	Now we show that $M'_p$ can be colored using $8t+5 = 4p+1$ colors in such a way that every vertex sees each color an odd number of times. We say that a vertex $v$ sees a color $i$, if there is an edge $e$ which is incident to $v$ and 
	$c(e) = i$. 
	
	Each copy $G_1^i$ can be constructed by considering the complete graph $K_{4p}$ with vertex set $\{0,1,2,\dots,4p-2,\infty\}$ and adding the edges of all following triangles:
	
	\begin{itemize}
		\item $(t+2+j),(t+3+j),(t+4+j)$ for every $j\in \{0,3,6,9,\dots,3(t-1)\}$;
		\item $-(t+2+j),-(t+3+j),-(t+4+j)$ for every $j\in \{0,3,6,9,\dots,3(t-1)\}$.
	\end{itemize}
	
	Consider the following $1$-factorization of $K_{4p}$.
	Let the edges of color $0$ be all edges of the set $M_0 = \{ 0\infty \}\cup \{ -ii\colon i\in \Z_{8t+3} \}$ and the edges of color $j\in \{0,1,\dots, 8t-2\}$ be all edges of the set $M_j = M_0+j = \{j\infty\}\cup\{(-i+j)(i+j)\colon i\in \Z_{8t+3}\}$. We can color this way all copies $K_{4p}^i$ of $K_{4p}$ inside $G_1^i$. Notice that we have used $8t+3=4p-1$ colors so far.
	
	\begin{figure}
		\includegraphics[scale=0.4]{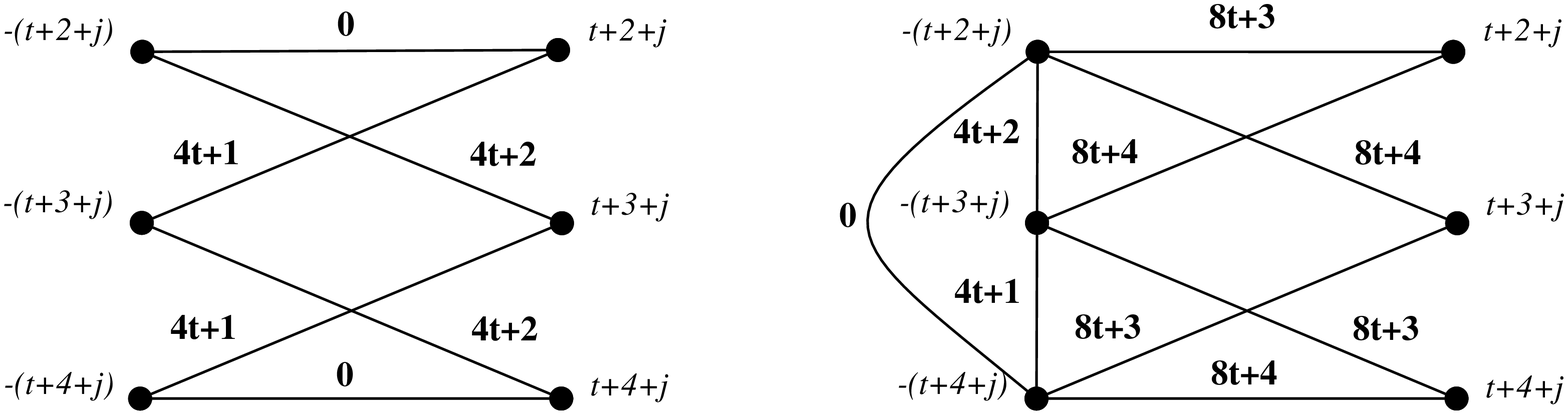}
		\caption{Color the added triangles using colors of the selected circuit. Color the selected circuit with two new colors. Colors are depicted in bold.}\label{Fig:triangles}
	\end{figure}
\begin{figure}
	\includegraphics[scale=0.5]{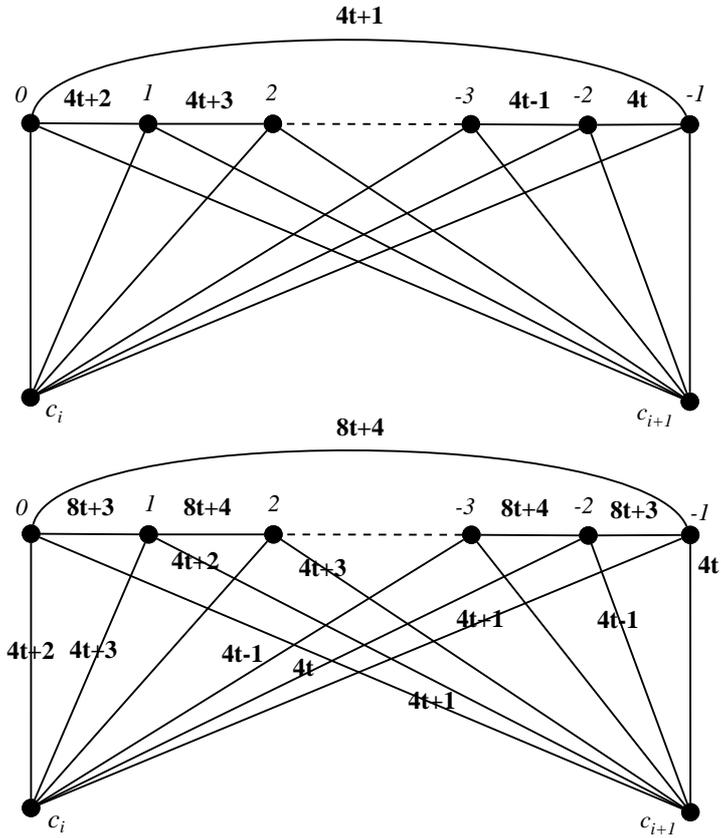}
	\caption{Assign to uncolored edges colors of the selected even circuit and assign to the edges of the selected circuit two new colors. Colors are depicted in bold.}\label{Fig:even_cycle}
\end{figure}
	
Consider the even circuits $ t+2+j,-(t+2+j),t+3+j,-(t+4+j),t+4+j,-(t+3+j),t+2+j$ for every $j\in\{0,3,6,9,\dots, 3(t-1)\}$ inside $G_1^i$. We perform the operation in Figure \ref{Fig:triangles} in order to color all triangles $(t+2+j),(t+3+j),(t+4+j)$ and $-(t+2+j),-(t+3+j),-(t+4+j)$ using two more colors $8t+3$ and $8t+4$.
	
	Consider the even circuit $C=0,1,2,\dots,t+1,\infty,-(t+1),-t,\dots,-2,-1$ inside $G_1^i$. Notice that these are all vertices of $K_{4p}^i$ in $G_1^i$ that are connected with both $c_i$ and $c_{i+1}$. Moreover, there are no two edges of $C$ belonging to the same color class. First assign colors $8t+3$ and $8t+4$ to the edges of $C$ alternately. This way we can assign the previous colors of the edges of $C$ to edges of the type $c_iv$ and $c_{i+1}v$, with $v\in V(C)$, in such a way that both $c_i$ and $c_{i+1}$ see different colors (notice that they see the very same set of colors), see Figure \ref{Fig:even_cycle}. Since the length of $C$ is $2t+4$, up to a permutation of colors, we can suppose that $c_i$ receives colors $i+\{1,3,5,7,\dots, 4t+7\} = i + \{2j+1\colon j=0,1,2,\dots,2t+3\}$ from the copy $G_1^i$, where now we are doing sums modulo $8t+5 = 4p+1$.
	
At this point, every vertex not in $\{w\} \cup \{c_i : i \in \{1, \dots, 4p+1\}\}$ sees every color exactly once.
	
	For every $i\in \{1,\dots,8t+5\}$, color all $p-2$ parallel edges connecting $c_i$ with $c_{i+1}$ with colors $i+\{4t+8,4t+10,\dots,8t+2,8t+4\}$, where sums are taken modulo $8t+5$ (notice that these are exactly $2t-1 (=p-2)$ colors). Finally, color $c_iw$ with color $i+4t+7$ modulo $8t+5$. The central vertex $w$ sees each color exactly once, whereas the vertex $c_i$ sees all colors once but for $i+4t+7$ which is seen three times.
\end{proof}

At this point we are going to further modify $M_p'$ in order to obtain a $(4p+1)$-regular graph of class $1$.

\subsection*{Construction of $\tilde{M_p}$}

Consider the graph $M_p'$. By Lemma \ref{lem:M_p'}, there is a $(4p+1)$-edge-coloring $c$ such that $|c^{-1}(i)\cap \partial (v)|$ is odd for every color $i \in \{0, \dots, 4p\}$ and vertex $v$. Construct $\tilde{M_p}$  by expanding each $c_i$ into a vertex of degree $4p+1$ that receives all colors and into a vertex of degree $2$ that receives the same color from both its adjacent edges. Finally suppress all such vertices of degree $2$.

\begin{teo}
	$\tilde{M_p}$ is a $(4p+1)$-regular class $1$ graph such that $\phi_c(\tilde{M_p})> 2+\frac{1}{p}$.
\end{teo}

\begin{proof}
	Since expanding and suppressing vertices does not decrease the circular flow number of graphs $\phi_c(\tilde{M_p})\ge \phi_c(M_p')> 2+\frac{1}{p}$. Furthermore a natural $(4p+1)$-edge-coloring is defined on $\tilde{M_p}$ by the edge-coloring of $M_p'$.
\end{proof}

\begin{cor}
	Let $p\ge3$ be any odd integer and let $k=2p$. There is a $(2k+1)$-regular class $1$ graph $G$ such that $\phi_c(G)>2+\frac{2}{k}$.
\end{cor}

\section*{Appendix}

\begin{proof}[Proof of Lemma \ref{lem:Flower}]
	We use induction on $n$. The statement holds true for $n=1$. Let $n\ge2$.
	For all $i\in \Z_{2n+1}$, let $J_{2n+1}^i$ be $G[\{a_i,b_i,c_i,d_i\}]$ and let $E_{i,i+1}=E(J_{2n+1}^i,J_{2n+1}^{i+1}) = \{a_ia_{i+1},c_id_{i+1},c_{i+1}d_i \}$. Suppose that there is $i\in \Z_{2n+1}$ such that $E_{i,i+1}\cap M = \emptyset$. Then it follows that $E_{i+2,i+3}\cap M = \emptyset$ as well. Remove $J_{2n+1}^{i+1}$ and $J_{2n+1}^{i+2}$ from $J_{2n+1}$ and add the edges $\{a_{i}a_{i+3},c_{i}d_{i+3},c_{i+3}d_{i} \}$. This new graph $H$ is isomorphic to the Flower snark $J_{2n-1}$ and so $H+M'$ is class $1$, where $M' = M \setminus (E(J_{2n+1}^{i+1}) \cup E(J_{2n+1}^{i+2}) \cup E_{i+1,i+2})$. Therefore, a natural $4$-edge-coloring $c$ is defined on $E(J_{2n+1}+M)\setminus(E(J_{2n+1}^{i+1}) \cup E(J_{2n+1}^{i+2}) \cup E_{i+1,i+2})$, such that $c(a_{i}a_{i+1})=c(a_{i+2}a_{i+3}), c(d_{i}c_{i+1})=c(d_{i+2}c_{i+3})$ and $c(c_{i}d_{i+1})=c(c_{i+2}d_{i+3})$. Let 
	\begin{itemize}
		\item $h_1 = c(c_{i}d_{i+1})=c(c_{i+2}d_{i+3})$;
		\item $h_2 = c(d_{i}c_{i+1})=c(d_{i+2}c_{i+3})$;
		\item $h_3 = c(a_{i}a_{i+1})=c(a_{i+2}a_{i+3})$.
	\end{itemize} Since $M\cap E_{i,i+1}$ is empty, $|E_{i+1,i+2}\cap M|=2 $ and we can assume without loss of generality that $E_{i+1,i+2}\cap M=\{c_id_{i+1},c_{i+1}d_i\}$ (and so $(E(J_{2n+1}^{i+1})\cup E(J_{2n+1}^{i+2})) \cap M=\{a_{i+1}b_{i+1},a_{i+2}b_{i+2}\}$). Furthermore it cannot happen that $h_1=h_2=h_3$. Indeed, we can assume w.l.o.g. that $M\cap E(J_{2n+1}^i) = \{b_{i}c_{i}\}$. Let $h_4 = c(b_id_i)$ and $h_5 = c(b_ia_i)$. And so either $h_1=h_4\ne h_2$ or $h_1=h_5\ne h_3.$ Now we extend the coloring on $J_{2n+1}$ to a proper $4$-edge-coloring.

	Consider the following auxiliary graph $G' = G[V(J_{2n+1}^{i+1})\cup V(J_{2n+1}^{i+2})]+M'$, where $M'=((E(G[V(J_{2n+1}^{i+1})\cup V(J_{2n+1}^{i+2})]) \cap M) \cup \{a_{i+1}a_{i+2},c_{i+1}d_{i+2},c_{i+2}d_{i+1}\}$. Figure \ref{Fig:coloring_flower} represents $G'$. Extending $c$ to a proper $4$-edge-coloring of $J_{2n+1}$ is equivalent to finding a proper edge-coloring of $G'$ for all possible $h_1 = c(c_{i+1}d_{i+2}),h_2=c(c_{i+2}d_{i+1}),h_3=c(a_{i+1}a_{i+2})$ non pairwise equal. Such a coloring is depicted in Figure \ref{Fig:coloring_flower}.
	
	Now we can assume that, for all $i\in\Z_{2n+1}$, $E_{i,i+1}\cap M \ne \emptyset$. In particular $|E_{i,i+1}\cap M|=1$. Indeed $|E_{i,i+1}\cap M|\ne E_{i,i+1}$ for otherwise the vertices $b_i$ and $b_{i+1}$ would not be matched. On the other hand, if $|E_{i,i+1}\cap M|=2$, then $E_{i-1,i}\cap M$ and $E_{i+1,i+2}\cap M$ would both be empty, a case that we already discussed.
	
	Define the following function $$t(E_{i,i+1}) = \begin{cases} 
	
	x_1=(1,0,0) & \mbox{if } E_{i,i+1}\cap M = \{c_id_{i+1}\} \\
	x_2=(0,1,0) & \mbox{if } E_{i,i+1}\cap M = \{c_{i+1}d_{i}\} \\
	x_3=(0,0,1) &\mbox{if } E_{i,i+1}\cap M = \{a_ia_{i+1}\}
	
\end{cases}$$

\begin{claim}
	There is $j \in\Z_{2n+1}$ such that $t(E_{j,j+1}) = t(E_{j+2,j+3})$.
\end{claim}

\begin{proof}[Proof of the Claim.]
	Notice that, for all $i\in \Z_{2n+1}$, \begin{itemize}
		\item if $t(E_{i,i+1})=x_1$, then $t(E_{i+1,i+2}),t(E_{i-1,i})\in \{x_1,x_3\}$;
		\item if $t(E_{i,i+1})=x_2$, then $t(E_{i+1,i+2}),t(E_{i-1,i})\in \{x_2,x_3\}$;
		\item if $t(E_{i,i+1})=x_3$, then $t(E_{i+1,i+2}),t(E_{i-1,i})\in \{x_1,x_2\}$.
	\end{itemize}
	Consider the circuit $C_{2n+1}$ on $2n+1$ vertices, let $E(C_{2n+1})=\{e_1,e_2,\dots,e_{2n+1}\}$ such that for every $i$, $e_i$ is adjacent to $e_{i+1}$. Let $\tilde{c} \colon E(C_{2n+1}) \to \{1,2,3\} $ be a coloring such that there are no adjacent edges $e_i,e_{i+1}$ with either $\tilde{c}(e_i)=\tilde{c}(e_{i+1})=3$ or $\{\tilde{c}(e_{i}),\tilde{c}(e_{i+1})\} = \{1,2\} $. Proving the statement of the claim is equivalent to proving that there are two edges $e_{j},e_{j+2}\in E(C_{2n+1})$ such that $\tilde{c}(e_j)=\tilde{c}(e_{j+2})$.
	
	Suppose by contradiction that there are not such edges. Then edges of color $3$ must be adjacent to exactly one edge of color $2$ and one of color $1$. On the other hand, edges of color $2$ and, respectively $1$, must be adjacent to exactly one edge of color $3$ and one of color $2$, respectively $1$.
	Let $$m_s = \begin{cases} 
	\text{number of pairs of adjacent edges of color } s & \mbox{if } s \in \{1,2\} \\
	\text{number of edges of color } s & \mbox{if } s = 3,
	
\end{cases}$$ then we can count the length of $C_{2n+1}$ as follows $$2n+1 = 2m_1 + 2m_2 + m_3. $$
Thus, $m_3$ is an odd number and we conclude that there is a $j\in \Z_{2n+1}$ such that $\tilde{c}(e_{j+1})=3$ and $\tilde{c}(e_{j})=\tilde{c}(e_{j+2})\in \{1,2\}$, a contradiction.
\end{proof}

By the Claim, there is $j$ such that $t(E_{j,j+1}) = t(E_{j+2,j+3})$. The graph $K$ obtained from $J_{2n+1}$ by removing $J_{2n+1}^{j+1}$ and $J_{2n+1}^{j+2}$ and adding the edges $\{ a_{j}a_{j+3}$, $c_{j}d_{j+3}$, $c_{j+3}d_{j} \}$ is isomorphic to $J_{2n-1}$. Thus, $K+M''$ has a proper $4$-edge-coloring, where $M''= M \setminus (E(J_{2n+1}^{j+1}) \cup E(J_{2n+1}^{j+2}) \cup E_{j+1,j+2})$. Therefore, a natural proper $4$-edge-coloring is defined on $E(J_{2n+1}+M)\setminus(E(J_{2n+1}^{j+1}) \cup E(J_{2n+1}^{j+2}) \cup E_{j+1,j+2})$. Without loss of generality we assume $t(E_{j,j+1}) = t(E_{j+2,j+3})= x_1$. Then the edge-coloring can be extended to a proper $4$-edge-coloring of $J_{2n+1}$ just as previous case. \qedhere

\begin{figure}
	\centering
	\includegraphics[scale=0.8]{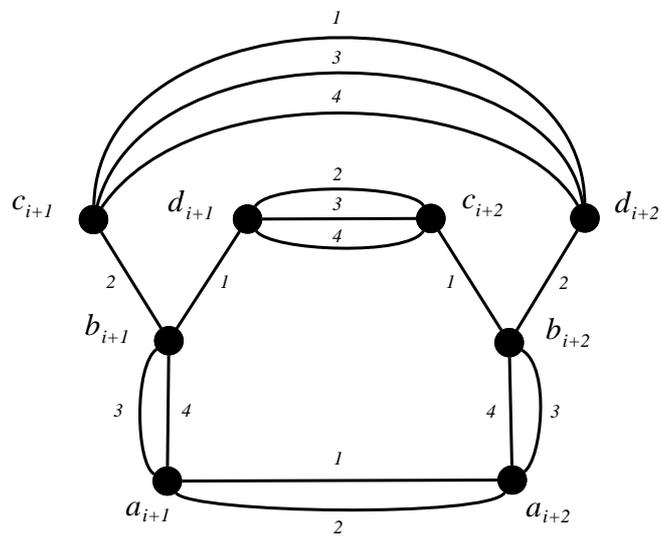}
	\caption{A coloring of the graph $G'$.}\label{Fig:coloring_flower}
\end{figure}

\end{proof}

\end{document}